\documentclass{amsart}[11pt]

\usepackage[utf8]{inputenc}
\usepackage{amsmath,amsthm,amsfonts,amssymb,graphicx,amscd}
\usepackage[all,cmtip]{xy}
\usepackage{tikz-cd}
\usepackage{graphicx}
\usepackage{stmaryrd}
\usepackage[utf8]{inputenc}
\usepackage[english]{babel}
\usepackage{tabstackengine}
\setstackgap{L}{.75\baselineskip}
\usepackage{tensor}
\usepackage{array}
\usepackage{hyperref}
\hypersetup{pdfborder=0 0 0}

\newcommand{\Hom}{\operatorname{Hom}}
\newcommand{\shom}{\operatorname{\underline{Hom}}}
\newcommand{\End}{\operatorname{End}}
\newcommand{\send}{\operatorname{\underline{End}}}
\newcommand{\smod}{\operatorname{\underline{mod}}}
\newcommand{\im}{\operatorname{Im}}
\newcommand{\Ext}{\operatorname{Ext}}
\newcommand{\soc}{\operatorname{soc}}

\newtheorem{thm}{Theorem}
\numberwithin{thm}{section}

\newtheorem{prop}{Proposition}
\numberwithin{prop}{section}

\newtheorem{rmk}{Remark}
\numberwithin{rmk}{section}



\begin{document}

\title[On universal deformation rings]{{\Large On universal deformation rings of modules over
a certain class of symmetric algebras of finite representation type}}

\author{Jhony F. Caranguay-Mainguez}
\address{Universidad de Antioquia, Instituto de Matem\'aticas}
\email{jhony.caranguay@udea.edu.co}

\author{Pedro Rizzo}
\address{Universidad de Antioquia, Instituto de Matem\'aticas}
\email{pedro.hernandez@udea.edu.co}

\author{Jos\'e A. V\'elez-Marulanda}
\address{Valdosta State University, Department of Applied Mathematics and Physics}
\email{javelezmarulanda@valdosta.edu}
\address{Fundaci\'on Universitaria Konrad Lorenz, Departamento de Matemáticas e Ingenier\'ias}
\email{josea.velezm@konradlorenz.edu.co}

\maketitle

\begin{abstract}
    Let $\Bbbk$ be an algebraically closed field. Recently, K. Erdmann classified the symmetric $\Bbbk$-algebras $\Lambda$ of finite representation type such that every non-projective module $M$ has period dividing four. The goal of this paper is to determine the indecomposable modules $M$ over these class of algebras $\Lambda$ whose stable endomorphism ring is isomorphic to $\Bbbk$, and then calculate their corresponding universal deformation rings (in the sense of F. M. Bleher and the third author).
\end{abstract}

\section{Introduction}

In this note, we are considering periodic finite-dimensional algebras. The classification of such algebras has increased lately among many authors. In particular, Erdmann et al. in \cite{ehs} discussed
\textit{tame symmetric algebras of period four} in order to achieve a general classification of such algebras. Moreover, K. Erdmann in \cite{kean} characterized symmetric algebras of finite type that have the property that every non-projective indecomposable module has $\Omega$-period dividing four. These classes of algebras include certain types of group algebras, weight surface algebras as well as their generalizations (i.e. the weighted generalized triangulation algebras). We refer to reader to \cite{ks1}, \cite{ks2}, \cite{ss} for more details.

On the other hand, the inspiration and foundation for the generalization of deformation theory (à la Mazur) for finitely generated modules over algebras of finite dimension can be traced back to representations of group algebras coming from finite groups. More precisely, in \cite{bv}, F. M. Bleher and the third author developed a deformation theory of finitely generated modules over finite-dimensional algebras, which was inspired by Mazur's deformation theory of Galois representations for profinite groups. They in particularly focused on self-injective algebras. This work carried significant implications for Frobenius algebras, which constitute an essential subclass of self-injective algebras and which also include the class of symmetric algebras.

The main objective of this paper is to determine the $\Lambda$-modules $M$ over the algebras $\Lambda$ characterized in \cite{kean} whose stable endomorphism ring is isomorphic to $\Bbbk$, and then calculate the corresponding universal deformation ring (i.e. apply \cite[Thm. 1.1]{bv}). To achieve this, we use results from the representation theory of finite-dimensional algebras, such as Auslander–Reiten quivers, stable equivalences, canonical homomorphisms and combinatorial descriptions of modules. Additionally, in our calculations we incorporate recent techniques developed in \cite{rv}.

On the other hand, F. M. Bleher and D. Wackwitz obtained in \cite{bw} equivalent results to the ones obtained in this paper by using stable equivalences of Morita type (in the sense of Brou\'e) induced by derived equivalences between self-injective algebras. More precisely, they use derived equivalences between Brauer tree algebras and certain class of symmetric Nakayama algebras (for more details see Remark \ref{rmk:1}). Our contribution, however, focuses on a direct and explicit calculation for these specific cases, avoiding the detour through derived equivalences. It is noteworthy that in \cite{crv} the authors use derived equivalences as a crucial tool to calculate the universal deformation rings for a certain family of modules over a wider class of Brauer graph algebras (in the sense of \cite{sib}), namely, the generalized Brauer tree algebras (in the sense of \cite{msw}).

The paper's main result is the following, where the algebras (AE1), (AE2) or (AE3) correspond to K. Erdmann classification in \cite{kean} (see Section \ref{sec:pre}) for a symmetric $\Bbbk$-algebra of finite representation type with period $4$.

\begin{thm}\label{thm:main}
Let $M$ be a non-projective indecomposable module over $\Lambda$.
\begin{enumerate}
\item If $\Lambda$ is the algebra (AE1), then $\send_{\Lambda}(M)\cong \Bbbk$ if and only if $M$ is in the $\Omega$-orbit of a module in the mouth of the unique exceptional tube of ${_s}\Gamma_{\Lambda}$, namely, $M\cong V_0$ or $M\cong V_{m-1}$. In this situation, we have $\Ext_{\Lambda}^1(M,M)=\Bbbk$ and $R(\Lambda,M)\cong \Bbbk \llbracket x\rrbracket/\langle x^{m+1}\rangle$.

\item If $\Lambda$ is the algebra (AE2), then the following statements hold.
\begin{enumerate}
    \item[$i)$] $\send_{\Lambda}(M)\simeq \Bbbk$ if and only if $M\cong M_j$ or $M\cong N_j$ for some $j\in \{0,1,2m-2,2m-1\}$.
    \item[$ii)$] Suppose that $\send_{\Lambda}(M)\cong \Bbbk$. Let $j$ be the element of $\{0,1,2m-2,2m-1\}$ such that $M\cong M_j$. Then,
\begin{equation*}
\Ext^1_{\Lambda}(M,M)= \begin{cases}
0 & \text{if $j\in \{0,2m-1\}$,}\\
\Bbbk & \text{if $j\in \{1,2m-2\}$,}
\end{cases}
\end{equation*}
     and the universal deformation ring of $M$ is given by
\begin{equation*}
R(\Lambda,M)= \begin{cases}
\Bbbk & \text{if $j\in \{0,2m-1\}$,}\\
\Bbbk \llbracket x \rrbracket /\langle x^m\rangle & \text{if $j\in \{1,2m-2\}$.}
\end{cases}
\end{equation*}

\end{enumerate}

\item If $\Lambda$ is the algebra (AE3), then the following statements hold.
\begin{enumerate}
    \item[$i)$] $\send_{\Lambda}(M)\simeq \Bbbk$ if and only if $M$ is isomorphic of one of the modules $U_0$, $V_1$, $X_m$, $Y_m$, $U_{m-1}$, $V_m$, $X_1$, $Y_1$.
    \item[$ii)$] Suppose that $\send_{\Lambda}(M)\simeq \Bbbk$. Then,
\begin{equation*}
\Ext^1_{\Lambda}(M,M)= \begin{cases}
0 & \text{if $M$ is isomorphic to one of the modules $U_0$, $V_1$, $X_m$, $Y_m$,}\\
\Bbbk & \text{if $M$ is isomorphic to one of the modules $U_{m-1}$, $V_m$, $X_1$, $Y_1$,}
\end{cases}
\end{equation*}
     and the universal deformation ring of $M$ is given by
\begin{equation*}
R(\Lambda,M)= \begin{cases}
\Bbbk & \text{if $M$ is isomorphic to one of the modules $U_0$, $V_1$, $X_m$, $Y_m$,}\\
\Bbbk \llbracket x \rrbracket /\langle x^m\rangle & \text{if $M$ is isomorphic to one of the modules $U_{m-1}$, $V_m$, $X_1$, $Y_1$.}
\end{cases}
\end{equation*}

\end{enumerate}

\end{enumerate}
\end{thm}

This paper is organized as follows. In Section \ref{sec:pre} we introduce the basic material and technical framework which allow us to establish in the next section a detailed description of our main calculations. In Section \ref{sec:main}, we will prove each part of Theorem~\ref{thm:main} separately: part (1) will be proven in Proposition~\ref{tmmod}, part (2) in Proposition~\ref{tc2eser}, and part (3) in Proposition~\ref{terd3}.

This article constitutes part of the development of the doctoral dissertation of the first author under the supervision of the other two authors.

\section{Preliminaries}\label{sec:pre}
In this section we introduce the basic material to achieve the main results and to facilitate the accompaniment of the reader in our calculations which will be presented in the next section.

Let $\Bbbk$ be an algebraically closed field and $\Lambda$ a finite dimensional $\Bbbk$-algebra. We assume that all the $\Lambda$-modules are finitely generated modules and from the right side. 

\subsection{Over certain class of symmetric algebras}\label{subsec:class}
For all modules $V$ over $\Lambda$, the \textit{syzygy of} $V$, denoted by $\Omega(V)$, is the kernel of a projective cover of $V$. More generally, for a minimal projective resolution of $V$:
$$
\cdots\longrightarrow P_{n+1}\stackrel{d_{n+1}}{\longrightarrow} P_n\longrightarrow\cdots\longrightarrow P_1\stackrel{d_1}{\longrightarrow} P_0 \stackrel{d_0}{\longrightarrow} V\longrightarrow 0
$$
we define the \textit{$n$-th syzygy} of $V$ by $\Omega^n(V):=\ker d_{n-1}$.
We say that $V$ is \textit{periodic} if $\Omega^n (V)\cong V$, for some positive integer $n$. The smallest positive integer $d$ with $\Omega^d (V)\cong V$ is called the \textit{period of} $V$. On the other hand, the $\Bbbk$-algebra $\Lambda$ is \textit{periodic} if it is periodic as a $\Lambda^e$-module, where $\Lambda^e$ is the \textit{enveloping algebra} of $\Lambda$, i.e. the tensor product $\Lambda^{op}\otimes_{\Bbbk} \Lambda$. For more details about periodic algebras and periodic modules see \cite{s}.

Assume that $\Lambda$ is basic and connected. In \cite{kean}, K. Erdmann proved that $\Lambda$ is a symmetric $\Bbbk$-algebra of finite representation type with period $4$ if and only if it is isomorphic to one of the following algebras.

\begin{enumerate}
    \item[\textbf{(AE1)}] $\Bbbk[x]/\langle x^m\rangle$, where $m\in \mathbb{Z}^+$;
    \item[\textbf{(AE2)}] the bound quiver algebra $\Bbbk Q/I$, where
$$
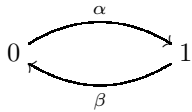
\xymatrix{
0 \ar@/^1pc/[rr]^{\alpha}
& & 1 \ar@/^1pc/[ll]^{\beta}}
$$
and $I=\langle (\alpha \beta)^m\alpha,(\beta\alpha)^m\beta \rangle$ for some $m\geq 1$;
    \item[\textbf{(AE3)}] the bound quiver algebra $\Bbbk Q/I$, where
$$
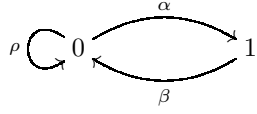
\xymatrix{
0 \ar@(ul,dl)[]_{\rho} \ar@/^1pc/[rr]^{\alpha}
&& 1 \ar@/^1pc/[ll]^{\beta}}
$$
and $I=\langle \rho\alpha,\beta\rho,\alpha\beta-\rho^m \rangle$ for some $m\geq 2$.
\end{enumerate}

\subsection{An overview of deformation theory of finitely-generated modules over finite dimensional algebras}\label{subsec:def}
Following \cite{bv}, we denote by $\widehat{\mathcal{C}}$ the category whose objects are commutative complete local Noetherian $\Bbbk$-algebras with residue field $\Bbbk$. Let $V$ and $W$ be $\Lambda$-modules. We denote the \textit{space of stable homomorphisms} by $\shom_{\Lambda}(V,W)$, which is defined to be the quotient $\Hom_{\Lambda}(V,W)/\mathcal{P}_{\Lambda}(V,W)$, where $\mathcal{P}_{\Lambda}(V,W)$ is the $\Bbbk$-vector subspace of $\Hom_{\Lambda}(V,W)$ generated by all the module homomorphisms that factor through projective modules. In particular, we denote by $\End_{\Lambda}(V)=\Hom_{\Lambda}(V,V)$ the endomorphism ring of $V$ and by $\send_{\Lambda}(V)$ the \textit{stable endomorphism ring} of $V$. 

Let $R\in\text{Obj}(\widehat{\mathcal{C}})$. A \emph{lift of $V$ over $R$} is a pair $(M,\phi)$, where $M$ is a finitely generated right $R\otimes_{\Bbbk}\Lambda$-module which is free over $R$, and $\phi:\Bbbk\otimes_R M\rightarrow V$ is an isomorphism of $\Lambda$-modules.
An isomorphism $f:(M,\phi)\rightarrow(M',\phi')$ between lifts is an isomorphism $f:M\rightarrow M'$ of $R\otimes_{\Bbbk}\Lambda$-modules such that $\phi'\circ (1_{\Bbbk}\otimes f)=\phi$, where $1_{\Bbbk}$ denotes the identity over $\Bbbk$. A \emph{deformation of $V$ over $R$} is defined to be an isomorphism class of lifts $[(M,\phi)]$ of $V$ over $R$. We denote by $\text{Def}_{\Lambda}(V,R)$ the set of all deformations of $V$ over $R$. This set has a functorial behavior, that is, we define the covariant \emph{deformation functor} $\widehat{F}_V:\widehat{\mathcal{C}}\rightarrow\textbf{Sets}$ as follows: for any $R\in \text{Obj}(\widehat{\mathcal{C}})$, we put $\widehat{F}_V(R):=\text{Def}_{\Lambda}(V,R)$ and for any morphism $\alpha:R\rightarrow R'$ in $\widehat{\mathcal{C}}$, the function $\widehat{F}_V(\alpha):\text{Def}_{\Lambda}(V,R)\rightarrow \text{Def}_{\Lambda}(V,R')$ is defined by $\widehat{F}_V(\alpha)([(M,\phi)])=[(R'\otimes_{R,\alpha}M,\phi_{\alpha})]$ where $\phi_{\alpha}:\Bbbk\otimes_{R'}(R'\otimes_{R,\alpha}M)\rightarrow V$ is obtained by the composition of isomorphisms $\Bbbk\otimes_{R'}(R'\otimes_{R,\alpha}M)\simeq \Bbbk\otimes_{R}M\stackrel{\phi}{\rightarrow} V$. In the case that there exists a unique object $R(\Lambda,V)$ in $\widehat{\mathcal{C}}$ such that $\widehat{F}_V$ is naturally isomorphic to the functor $\text{Hom}_{\widehat{\mathcal{C}}}(R(\Lambda,V),-)$, i.e. $R(\Lambda,V)$ represents $\widehat{F}_V$, we say that $R(\Lambda,V)$ is the \emph{universal deformation ring} of $V$. An important consequence under this situation is that there exists a deformation $[(U(\Lambda, V),\phi_{U(\Lambda, V)})]$, which we call the \emph{universal deformation of $V$}, such that, given any $R\in \text{Obj}(\widehat{\mathcal{C}})$ and any deformation $[(M,\phi)]$ of $V$, there exists a unique morphism $\alpha:R(\Lambda, V)\rightarrow R$ such that $\widehat{F}_V(\alpha)([\left(U(\Lambda, V),\phi_{U(\Lambda, V)}\right)])=[(M,\phi)]$. 

Moreover, the \textit{tangent space} $t_V$ \textit{of} $F_V$ is $t_V=F_V(\Bbbk\llbracket x \rrbracket/\langle x^2\rangle)$. It follows from \cite[Lemma 2.10]{schl} that $t_V$ is a vector space, which is isomorphic to the first group of extensions $\Ext_{\Lambda}^1(V,V)$ (see \cite[Proposition 2.1]{bv}). Besides, if $V$ admits a universal deformation ring then, $R(\Lambda,V)=\Bbbk$ if $t_V=0$ and $R(\Lambda,V)$ is a quotient of $\Bbbk \llbracket x_1,\ldots,x_{\dim_{\Bbbk}(t_V)}\rrbracket$ if $\dim_{\Bbbk}(t_V)>0$ (see \cite[pg. 223]{bw}).

Recall that $\Lambda$ is \emph{self-injective} if the right $\Lambda$-module $\Lambda_{\Lambda}$ is injective. On the other hand, we say that $\Lambda$ is \emph{Frobenius} if there exists an isomorphism as right $\Lambda$-modules between $\Lambda_{\Lambda}$ and $D(\Lambda)_{\Lambda}$, where $D(\Lambda)=\text{Hom}_{\Bbbk}(\Lambda_{\Lambda},\Bbbk)$. Recall also that $\Lambda$ is called \emph{symmetric} algebra if $\Lambda$ is endowed with a non-degenerate, associative bilinear form $B:\Lambda\times\Lambda\rightarrow\Bbbk$ such that $B(a,b)=B(b,a)$, for all $a,b\in\Lambda$. By \cite[Chapter IV, Proposition 3.8]{sy}, any Frobenius algebra is self-injective and every symmetric algebra is a Frobenius algebra.

The following theorem follows from \cite[Theorem 2.6]{bv}.
\begin{thm}\label{thm:defor}
Assume that $\Lambda$ is self-injective and suppose that $V$ is a finitely generated $\Lambda$-module whose stable endomorphism ring $\send_{\Lambda}(V)$ is isomorphic to $\Bbbk$. Then, the following conditions hold:
\begin{enumerate}
\item[i)] The module $V$ has a universal deformation ring $R(\Lambda, V)$.
\item[ii)] If $P$ is a finitely generated projective $\Lambda$-module, then $\send_{\Lambda}(V\oplus P)\simeq\Bbbk$ and $R(\Lambda, V)=R(\Lambda,V\oplus P)$.
\item[iii)] If $\Lambda$ is a Frobenius algebra, then $\send_{\Lambda}(\Omega(V))\simeq\Bbbk$ and $R(\Lambda, V)=R(\Lambda,\Omega(V))$, where $\Omega(V)$ is the first syzygy of $V$.
\end{enumerate}
\end{thm}

We will use the following theorem, which is just a particular case of a result given due to the second and third authors (see  \cite[Theorem 1.1]{rv}).

\begin{thm}\label{anote} Suppose that $V$ admits a universal deformation ring $R(\Lambda,V)$, which is a quotient of the ring of formal series $\Bbbk \llbracket x \rrbracket$. Moreover, assume that there exists a sequence $\mathcal{L}=\{V_0,V_1,\ldots\}$, where $V_0=V$ and such that for $l>0$ there are a monomorphism $\iota_l:V_{l-1}\rightarrow V_l$ and an epimorphism $\epsilon:V_l\rightarrow V_{l-1}$ for which the composition $\sigma_l:=\iota_l\epsilon_l$ satisfies that $\ker(\sigma_l)\cong V_0$, $\im(\sigma_l^l)\cong V_0$ and $\mathcal{L}$ is maximal respect with these properties.
\begin{enumerate}
    \item[i)] If $\mathcal{L}$ is finite with $\mathcal{L}=\{V_0,V_1,\ldots,V_N\}$, $\dim_{\Bbbk} \Hom_{\Lambda}(V_N,V)=1$ and $\Ext_{\Lambda}^1(V_N,V)=0$, then we have $R(\Lambda,V)\cong \Bbbk\llbracket x\rrbracket/\langle x^{N+1}\rangle$.
    \item[ii)] If $\mathcal{L}$ is infinite, then $R(\Lambda,V)\cong \Bbbk \llbracket x \rrbracket$.
\end{enumerate}
\end{thm}

\subsection{A quick tour about the stable Auslander-Reiten quiver and homomorphisms between string modules}\label{subsec:string}
Assume that $\Lambda=\Bbbk Q/I$ is a bound quiver algebra on a quiver $Q=(Q_0,Q_1,s,t)$, where $I$ is an admisible ideal of $\Bbbk Q$ (see e.g. \cite[Chapter III]{ars}). From now on, we will assume that $\Lambda$ is further a symmetric special biserial algebra (see e.g. \cite[Chapter II]{kebt}). We denote by ${}_{s}\Gamma_{\Lambda}$ the stable Auslander-Reiten quiver of $\Lambda$. 

For any arrow $\alpha\in Q_1$ we denote by $\alpha^{-1}$ the \textit{formal inverse} of $\alpha$. We define a \textit{word of length} $n\geq1$ as the sequence $w_1\cdots w_n$, where the $w_i$ represents either an arrow or the formal inverse of an arrow, with the condition $s(w_{j+1})=t(w_j)$, where $j\in\{1,\ldots,n-1\}$. We define $s(w_1\cdots w_n)=s(w_1)$, $t(w_1\cdots w_n)=t(w_n)$ and $(w_1\cdots w_n)^{-1}=w_n^{-1}\cdots w_1^{-1}$. On the other hand, we define for any vertex $i\in Q_0$ the \textit{empty word} $\boldsymbol{e_i}$ as the word of length 0, such that $s(\boldsymbol{e_i})=t(\boldsymbol{e_i})=i$ and $\boldsymbol{e_i}^{-1}=\boldsymbol{e_i}$. We denote by $\mathcal{W}$ the set of all words in $Q$. Let ``$\sim$'' be the equivalence relation on $\mathcal{W}$ defined by $w\sim w'$ if and only if $w=w'$ of $w^{-1}=w'$. A \textit{string} $C$ is a repesentative of an equivalence class under the equivalence relation $\sim$, which satisfies that either $C=\boldsymbol{e_i}$ for some vertex $i\in Q_0$ or $C=w_1w_{2}\cdots w_n$ with $n\geq1$, $w_{j}\neq w_{j+1}^{-1}$ for $1\leq j\leq n-1$ and no sub-word of $C$ or its formal inverse belongs to $I$. Now, for any pair of strings $C=w_1w_{2}\cdots w_n$ and $D=v_1v_{2}\cdots v_m$ of length $n$ and $m$, respectively, we define the \textit{composition of} $C$ \textit{and} $D$ as the concatenation $CD=w_1\cdots w_nv_1v_2\cdots v_m$ provided that it is a string. Moreover, we set $e_{s(C)}C=Ce_{t(C)}=C$. If $C=w_1\cdots w_n$ is a string of length $n\geq1$, then there exists an indecomposable $\Lambda$-module $M[C]$, called \textit{the string module corresponding to the string representative} $C$ (see e.g. \cite[Section II.3]{kebt}). For every string $C$, it holds that $M[C]\simeq M[C^{-1}]$; if $C=\boldsymbol{e_i}$, then $M[C]=S(i)$ is the simple $\Lambda$-module associated to the vertex $i$.

Assume $C=w_1\cdots w_n$ is a string of length $n\geq1$. We call $C$ a \textit{directed string} if all $w_j$ are arrows, for $1\leq j\leq n$. We say that $C$ is a \textit{maximal directed string} if $C$ is a directed string and $\alpha C\in I$, for any arrow $\alpha\in Q_1$. We denote by $\mathcal{M}$ the set of all maximal directed strings. 

Following e.g. \cite[Section II.5]{kebt}, we say that $C$ \textit{starts on a peak} (resp. \textit{starts in a deep}, \textit{ends on a peak}, \textit{ends in a deep}) if there is no arrow $\alpha\in Q_1$ such that $\alpha C$ (resp. $\alpha^{-1}C$, $C\alpha^{-1}$, $C\alpha$ ) is a string. If $C$ is a string not starting on a peak (resp. not starting in a deep), that is, $\alpha C$ (resp. $\alpha^{-1}C$) is a string for some $\alpha\in Q_1$, then there exists a unique $D\in\mathcal{M}$ such that $C_h:= D^{-1}\alpha C$ (resp. $C_c:= D\alpha^{-1}C$) is a string. We say that $C_h$ (resp. $C_c$) is obtained by $C$ by adding a \textit{hook} (resp. a \textit{co-hook}) \textit{on the left side}. Dually, if $C$ is a string not ending on a peak (resp. not ending in a deep), that is, $C\alpha^{-1}$ (resp. $C\alpha $) is a string for some $\alpha\in Q_1$, then there exists a unique $E\in \mathcal{M}$ such that ${}_{h}C=C\alpha^{-1}E$ (resp. ${}_{c}C=C\alpha E^{-1} $) is a string. We say that ${}_{h}C$ (resp. ${}_{c}C$) is obtained by $C$ by adding a \textit{hook} (resp. \textit{co-hook}) \textit{on the right side}. 

By \cite{br} (see also \cite[Proposition II.5.3]{kebt}), all irreducible morphisms between string modules are either canonical injections $M[C]\rightarrow M[C_h]$, $M[C]\rightarrow M[{}_{h}C]$, or canonical projections $M[C_c]\rightarrow M[C]$, $M[{}_{c}C]\rightarrow M[C]$. 

The homomorphisms between string modules were completely classified by H. Krause in \cite{kra}. We discuss such classification as follows. Let $S$ and $T$ be strings for $\Lambda$. Suppose that $C$ is a substring of both $S$ and $T$ such that the following conditions are satisfied:
\begin{enumerate}
\item[i)] $S\sim DCB$, where $B$ is a substring which is either of length zero or $B=\alpha B'$, for an arrow $\alpha$ and some string $B'$, and $D$ is a substring which is either of length zero or $D=D'\beta^{-1}$, for an arrow $\beta$ and some string $D'$.
\item[ii)] $T\sim FCE$, where $E$ is a substring which is either length zero or $E=\gamma^{-1}E'$ for an arrow $\gamma$ and some string $E'$, and $F$ is a substring which is either of length zero of $F=F'\lambda$, for an arrow $\lambda$ and some string $F'$.
\end{enumerate}
Then, there exists a composition of $\Lambda$-modules homomorphisms $\sigma_C:M[S]\twoheadrightarrow M[C]\hookrightarrow M[T]$. We call $\sigma_C$ a \textit{canonical homomorphism} from $M[S]$ to $M[T]$ that factors through $M[C]$. It follows from \cite{kra} that each $\Lambda$-module homomorphism from $M[S]$ to $M[T]$ can be written uniquely as a $\Bbbk$-linear combination of canonical homomorphisms. In particular, if $M[S]=M[T]$ then the canonical endomorphisms of $M[S]$ generate $\End_{\Lambda}(M[S])$.

\section{The main results}\label{sec:main}

In this section we assume that $\Lambda$ is either one of the algebras (AE1), (AE2) or (AE3). Two very important properties of $\Lambda$ are the following: Firstly, in general, if $\Lambda$ is a periodic finite dimensional algebra, then every non-projective module is periodic with period dividing the period of $\Lambda$. Therefore, the class of algebras satisfying this property encompasses the class of periodic algebras. Secondly, since $\Lambda$ is a special biserial algebra of finite representation type, all the components of ${_s}\Gamma_{\Lambda}$ are exceptional, i.e. they are composed entirely by string modules.

Recall that $\Omega$ represents the first syzygy operator. Since $\Lambda$ is symmetric, it follows that $\Omega$ induces a self-equivalence of $\smod$-$\Lambda$ (see \cite[IV. Proposition~3.8]{ars}) and hence $\send_{\Lambda}(M)\cong \send_{\Lambda}(\Omega(M))$ for all $\Lambda$-modules $M$.

\subsection{Case (AE1)}
Let $\Lambda=\Bbbk\llbracket x \rrbracket/\langle x^{m+1}\rangle$, where $m$ is a positive integer. This algebra also can be presented as the bound quiver algebra $\Bbbk Q/I$, where
$$
Q:= \xymatrix{
0 \ar@(ul,dl)[]_{\alpha}
}
$$
and $I=\langle \alpha^{m+1} \rangle$ for some $m\in\mathbb{Z}^{+}$. Note that the radical series of the unique projective indecomposable module corresponding to the vertex $0$ can be represented as follows:
\begin{equation*}
P(0)=
\begin{matrix}
S(0)\\
\vdots\\
S(0)
\end{matrix} 
\end{equation*}
\noindent
Moreover, we have that $\Lambda_s:=\Lambda/\soc(\Lambda) \cong \Bbbk Q/I_s$, where $I_s=\langle \alpha^{m} \rangle$. For $j\in \{0,\ldots,m-1\}$, we denote by $V_j$ the string module $M[C_j]$, where $C_j:=e_0$ if $j=0$, and $C_j:=\alpha^j$ for $1\leq j\leq m-1$.
%
The stable Auslander-Reiten quiver of $\Lambda$ is as in Figure \ref{Fig1}.

\begin{figure}[ht]
$\xymatrix{
V_0 \ar@/^1pc/[rr] & & V_1 \ar@/^1pc/[ll] \ar@/^1pc/[rr] & & \cdots \ar@/^1pc/[rr] \ar@/^1pc/[ll] & & V_{m-1} \ar@/^1pc/[ll]
}$
\caption{The stable Auslander-Reiten quiver for \textbf{(AE1)}.}
\label{Fig1}
\end{figure}
\noindent
In particular, note that $V_j={}_{h}(V_{j-1})$ and $V_j=(V_{j-1})_c$ for every $j\in \{1,\ldots,m-1\}$.

\begin{prop}\label{tmmod} 
Let $M$ be a non-projective indecomposable module over $\Lambda$. Then, $\send_{\Lambda}(M)\cong \Bbbk$ if and only if $M$ is in the $\Omega$-orbit of a module in the mouth of the unique exceptional tube of ${_s}\Gamma_{\Lambda}$, namely, $M\cong V_0$ or $M\cong V_{m-1}$. In this situation, we have $\Ext_{\Lambda}^1(M,M)=\Bbbk$ and $R(\Lambda,M)\cong \Bbbk \llbracket x\rrbracket/\langle x^{m+1}\rangle$.
\end{prop}

\begin{proof} \mbox{}\\
\textit{Step 1.} Observe that the modules in the mouth of ${}_{s}\Gamma_{\Lambda}$ are $V_0$ and $V_{m-1}$. By the isomorphism $\Omega(V_0)\cong V_{m-1}$ and by Schur's Lemma, we have that 
$$
\dim_{\Bbbk} \End_{\Lambda}(V_0)=\dim_{\Bbbk} \End_{\Lambda}(V_{m-1})=1,
$$
and therefore, $V_0$ and $V_{m-1}$ has stable endomorphism ring isomorphic to $\Bbbk$. Next, assume that $0<j<m-1$. To see that $V_j$ has no stable endomorphism ring isomorphic to $\Bbbk$, consider the canonical endomorphism $f:V_j\twoheadrightarrow V_{j-1}\hookrightarrow{} V_j$. Let $\phi$ and $\psi$ be canonical homomorphisms $\phi:V_j\twoheadrightarrow V_x\hookrightarrow{} P(0)$ and $\psi:P(0)\twoheadrightarrow V_y\hookrightarrow{} V_j$ such that $\psi \phi$ preserves the top submodule $V_{j-1}$ of $V_j$. It follows that $x=j$, because otherwise the image of $V_{j-1}$ under $\phi$ has a non-zero intersection with $\ker{\psi}$, which cannot occur. Moreover, since $\psi \phi$ preserves $V_{j-1}$ and by the composition series of the string module $M[C_j]$, we obtain that $y=m-1$, which is a contradiction. Hence, there are no homomorphisms in $\mathcal{P}_{\Lambda}(V_j,V_j)$ preserving $V_{j-1}$. So, $f\notin \mathcal{P}_{\Lambda}(V_j,V_j)$.

\textit{Step 2.} Note that the $\Bbbk$-vector space $\Hom_{\Lambda}(\Omega(V_0),V_0)=\Hom_{\Lambda}(V_{m-1},V_0)$ is generated by the canonical projection $V_{m-1}\twoheadrightarrow V_0$. By the length of the radical series of $P(0)$, we have that this projection does not belong to $\mathcal{P}_{\Lambda}(\Omega(V_0),V_0)$ and hence
$$\Ext_{\Lambda}^1(V_0,V_0)\cong \shom_{\Lambda}(\Omega(V_0),V_0) \cong \Bbbk.$$
On the other hand, let $\mathcal{L}:=\{V_0,\ldots,V_m\}$, and for $\ell \in \{1,\ldots,m\}$ define $\iota_{\ell}$ as the natural inclusion $V_{\ell-1}\hookrightarrow{} V_{\ell}$, $\epsilon_{\ell}$ as the canonical projection $V_{\ell}\twoheadrightarrow V_{\ell-1}$, and $\sigma_{\ell}:V_{\ell}\rightarrow V_{\ell}$ as the composition $\iota_{\ell}\epsilon_{\ell}$. Then, $\ker \sigma_{\ell}=V_0$ and $\im(\sigma_{\ell}^{\ell})=V_0$. Since $\Hom_{\Lambda}(V_m,V_0)$ is generated by the canonical projection and $\Ext_{\Lambda}^1(V_m,V_0)=0$ because $V_m$ is projective, it follows by Theorem \ref{anote} that $R(\Lambda,V_0)\cong \Bbbk \llbracket x\rrbracket/\langle x^{m+1}\rangle$.
\end{proof}

\subsection{Case (AE2)}


Let $\Lambda$ be the bound quiver algebra $\Bbbk Q/I$, where
$$
Q:= \xymatrix{
0 \ar@/^1pc/[rr]^{\alpha}
& & 1 \ar@/^1pc/[ll]^{\beta}}
$$
and $I=\langle (\alpha \beta)^m\alpha,(\beta\alpha)^m\beta \rangle$ for some $m\geq 1$. For this algebra, the radical series of the projective indecomposable modules have the following form.
\begin{equation*}
P(0)=
\begin{matrix}
S(0)\\
S(1)\\
\vdots\\
S(0)\\
S(1)\\
S(0)
\end{matrix}
\ \ \mbox{ and }\ \ \
P(1)=
\begin{matrix}
S(1)\\
S(0)\\
\vdots\\
S(1)\\
S(0)\\
S(1)
\end{matrix}
\end{equation*}

Set $\Lambda_s:=\Lambda/soc(\Lambda)$. Then, $\Lambda_s\cong \Bbbk Q/I_s$, where $I_s=\langle (\alpha \beta)^m,(\beta \alpha)^m \rangle$. Moreover, we have that ${}_{s}\Gamma_{\Lambda}$ coincides with the Auslander-Reiten quiver $\Gamma_{\Lambda_s}$. For $j\in \{0,\ldots,2m-1\}$ define $M_j$ and $N_j$ to be the string modules $M[C_j]$ and $M[D_j]$, respectively, where $C_j$ and $D_j$ are the string given as follows.
\begin{enumerate}
    \item[$i)$] $C_0=e_0$, $C_1=\alpha$.
    \item[$ii)$] $C_j=(\alpha \beta)^{j/2}$ if $1<j<2m$ and $j$ is even.
    \item[$iii)$] $C_j=(\alpha \beta)^{(j-1)/2}\alpha$ if $1<j<2m$ and $j$ is odd.
    \item[$iv)$] $D_0=e_1$, $D_1=\beta$.
    \item[$v)$] $D_j=(\beta \alpha)^{j/2}$ if $1<j<2m$ and $j$ is even.
    \item[$vi)$] $D_j=(\beta \alpha)^{(j-1)/2}\beta$ if $1<j<2m$ and $j$ is odd.
\end{enumerate}

The stable Auslander-Reiten quiver of $\Lambda$ consists of a unique connected component which is represented graphically in Figure \ref{fig:AR1},
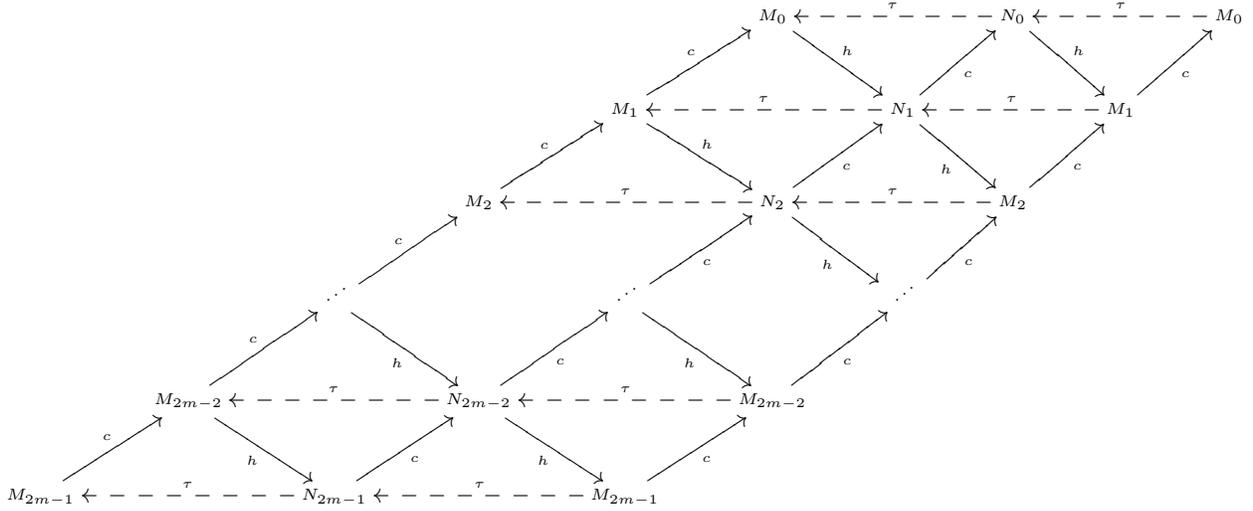
\begin{figure}
{\tiny
$$\xymatrix{
& & & & & M_0 \ar[rd]^h & & N_0 \ar[rd]^h \ar@{-->}[ll]_{\tau} & & M_0 \ar@{-->}[ll]_{\tau} \\
& & & & M_1 \ar[rd]^h \ar[ru]^c & & N_1 \ar[ru]_c \ar[rd]_h \ar@{-->}[ll]_{\tau} & & M_1 \ar[ru]_c \ar@{-->}[ll]_{\tau} & \\
& & & M_2 \ar[ru]^c & & N_2 \ar[ru]_c \ar[rd]_h \ar@{-->}[ll]_{\tau} & & M_2 \ar[ru]_c \ar@{-->}[ll]_{\tau} & & \\
& & \ \rotatebox{35}{$\cdots$} \ar[ru]^c \ar[rd]_h & & \ \rotatebox{35}{$\cdots$} \ar[rd]_h \ar[ru]_c & & \ \rotatebox{40}{$\cdots$} \ar[ru]_c & & & \\
& M_{2m-2} \ar[rd]_h \ar[ru]^c & & N_{2m-2} \ar[ru]_c \ar[rd]_h \ar@{-->}[ll]_{\tau} & & M_{2m-2} \ar[ru]_c \ar@{-->}[ll]_{\tau} & & & & \\
M_{2m-1}\ar[ru]^c & & N_{2m-1} \ar[ru]_c \ar@{-->}[ll]_{\tau} & & M_{2m-1} \ar[ru]_c \ar@{-->}[ll]_{\tau} & & & & & }$$
}
\caption{The stable Auslander-Reiten quiver for \textbf{(AE2)}.}
\label{fig:AR1}
\end{figure}
where it holds that $M_j={}_{c}(M_{j-1})$, $N_j={}_{c}(N_{j-1})$, $N_j=(M_{j-1})_h$, and $M_j=(N_{j-1})_h$ for all $j\in \{1,\ldots,2m-1\}$.

\begin{prop}\label{tc2eser}
Let $M$ be a non-projective indecomposable module over $\Lambda$. Then, the following statements hold.
\begin{enumerate}
    \item[$i)$] $\send_{\Lambda}(M)\simeq \Bbbk$ if and only if $M\cong M_j$ or $M\cong N_j$ for some $j\in \{0,1,2m-2,2m-1\}$.
    \item[$ii)$] Suppose that $\send_{\Lambda}(M)\cong \Bbbk$. Let $j$ be the element of $\{0,1,2m-2,2m-1\}$ such that $M\cong M_j$. Then,
\begin{equation*}
\Ext^1_{\Lambda}(M,M)= \begin{cases}
0 & \text{if $j\in \{0,2m-1\}$,}\\
\Bbbk & \text{if $j\in \{1,2m-2\}$,}
\end{cases}
\end{equation*}
     and the universal deformation ring of $M$ is given by
\begin{equation*}
R(\Lambda,M)= \begin{cases}
\Bbbk & \text{if $j\in \{0,2m-1\}$,}\\
\Bbbk \llbracket x \rrbracket /\langle x^m\rangle & \text{if $j\in \{1,2m-2\}$.}
\end{cases}
\end{equation*}

\end{enumerate}
\end{prop}

\begin{proof} \mbox{}\\
\textit{Step 1.} First, since the modules $M_0=S(0)$, $N_0=S(1)$, $M_{2m-1}$ and $N_{2m-1}$ are in the same $\Omega$-orbit. and also $M_1$, $N_1$, $M_{2m-2}$ and $N_{2m-2}$ are in a common $\Omega$-orbit, the stable endomorphism ring of each one of them is isomorphic to $\send_{\Lambda}(S(0))$ or $\send_{\Lambda}(M_1)$. By Schur's Lemma, $\send_{\Lambda}(S(0)) \cong \Bbbk$ and, due to there exist no canonical homomorphisms  from $M_1$ to $M_1$ different from the identity, we have $\send_{\Lambda}(M_1) \cong \Bbbk$.

Now, we prove that any other module that does not appear in the table has no stable endomorphism ring isomorphic to $\Bbbk$. Taking representatives of the $\Omega$-orbits, we verify that $\send_{\Lambda}(M_j)\not\cong \Bbbk$ for $j\neq 0,1,2m-1,2m-2$. In fact, for such values of $j$ write $C_j=E\gamma$, where $E$ and $\gamma$ are strings, and $\gamma$ is a path of length $2$. If $j$ is even, then $C_j=\gamma E$. If $j$ is odd, then $C_j=\delta E$, where $\delta=\alpha \beta$ when $\gamma=\beta \alpha$ and $\delta=\beta \alpha$ when $\gamma=\alpha \beta$. In any case, we obtain a canonical homomorphism $f:M_j\hookrightarrow{} M[E] \twoheadrightarrow M_j$. By the form of the projective indecomposable modules and the form of the strings $C_j$ for $j\neq 0,1,2m-1,2m-2$, there are no canonical homomorphisms $\varphi:M_j\rightarrow P(t)$ and $\psi:P(t)\rightarrow M_j$ such that $t\in \{0,1\}$ and $\psi\varphi$ preserves the module $M[E]$. Hence,  $\send_{\Lambda}(M_j)\not\cong \Bbbk$.

\noindent
\textit{Step 2.} For the simple module $S(0)$, we have
$$
\Ext_{\Lambda}^1(S(0),S(0))\cong \shom_{\Lambda}(\Omega(S(0)),S(0))=\shom_{\Lambda}(N_{2m-1},S(0))=0
$$
because there are no canonical homomorphisms from $N_{2m-1}$ to $S(0)$. This implies that $R(\Lambda,S(0))\cong \Bbbk$.

On the other hand, we have that the unique canonical homomorphism from $M_{2m-2}$ to $M_1$ is the canonical projection $\phi:M_{2m-2}\twoheadrightarrow M_1$. There are no canonical homomorphisms $\varphi:M_{2m-2}\rightarrow P$ and $\psi:P\rightarrow M_1$ for a projective indecomposable module $P$ such that $\psi \varphi$ preserves $M_1$. Then, $\phi \notin \mathcal{P}_{\Lambda}(M_{2m-2},M_1)$. Therefore,
$$
\Ext_{\Lambda}^1(M_1,M_1)\cong \shom_{\Lambda}(\Omega(M_1),M_1)=\shom_{\Lambda}(M_{2m-2},M_1)\cong \Bbbk.
$$
To calculate the universal deformation ring of $M_1$ we will be apply Theorem \ref{anote}. Consider the sequence of $\Lambda$-modules $\mathcal{L}:=\{V_0,\ldots,V_{m-1}\}$, where $V_l:=M_{2l+1}$ for $l\in \{0,\ldots,m-1\}$. Also, let $\iota_l:V_{l-1}\hookrightarrow{} V_l$ be the natural inclusion and $\epsilon_l:V_l\twoheadrightarrow V_{l-1}$ the canonical projection, and define $\sigma$ to be the composition $\iota_l\epsilon_l$, for each $l\in \{1,\ldots,m-1\}$. Then, we have $\ker(\sigma_l)=V_0$. Besides, $\im(\sigma_l^t)=V_{l-t}$ for $l\in \{1,\ldots,m-1\}$ and $t\in \{1,\ldots,l\}$. In particular, $\im(\sigma_l^l)=V_0$. Moreover, by a previous calculus, we get $\Hom_{\Lambda}(V_{m-1},V_0)\cong \Bbbk$ and, since there are no canonical homomorphisms from $S(0)$ to $V_1$, we have
$$
\Ext_{\Lambda}^1(V_{m-1},V_0)\cong \shom_{\Lambda}(\Omega(V_{m-1}),V_0)=\shom_{\Lambda}(S(0),V_0)=0.
$$
Thus, by Theorem \ref{anote}, we conclude that $R(\Lambda,M_1)\cong \Bbbk \llbracket x\rrbracket/\langle x^m\rangle$. Finally, the statement follows from Theorem \ref{thm:defor}.

\end{proof}

\subsection{Case (AE3)}

Let $\Lambda$ be the bound quiver algebra $\Bbbk Q/I$, where
$$Q:\ \xymatrix{
0 \ar@(ul,dl)[]_{\rho} \ar@/^1pc/[rr]^{\alpha} && 1 \ar@/^1pc/[ll]^{\beta}}
$$
and $I=\langle \rho\alpha,\beta\rho,\alpha\beta-\rho^m \rangle$ for some $m\geq 2$. We have that the radical series of the projective indecomposable $\Lambda$-modules are the following.
\begin{center}
$P(0)=$
\tabbedCenterstack{& $S(0)$ &\\
$S(0)$ & & \\
&&\\
\vdots & & $S(1)$\\
&&\\
$S(0)$ & & \\
&&\\
 & $S(0)$ &}
\ \ \qquad and\qquad \ \ $P(1)=$
\tabbedCenterstack{$S(1)$\\
$S(0)$\\
$S(1)$}
\end{center}
and hence $\Lambda_s:=\Lambda/soc(\Lambda)\cong \Bbbk Q/I_s$, where $I_s=\langle \alpha\beta,\beta\alpha,\rho^m,\beta\rho,\rho\alpha \rangle$.

Since $\Lambda$ is a self-injective algebra, we obtain the stable Auslander-Reiten quiver $\Gamma_s(\Lambda)$ from the Auslander-Reiten quiver $\Gamma_{\Lambda_s}$.

A complete list of strings of $\Lambda_s$ is given by $\rho^{i-m}$, $\rho^{i-m}\alpha$ and $\beta\rho^{i-m}$, where $i\in \{1,\ldots,m\}$, joint with $\beta\rho^{i-m}\alpha$, where $i\in \{1,\ldots,m-1\}$. Here, $\rho^0:=e_0$. Now, we define the following list of string modules over $\Lambda_s$:

\begin{enumerate}
    \item[$i)$] $V_i=M[\rho^{i-m}]$, with $i\in \{1,\ldots,m\}$. Notice that $V_m=S(0)$.
    
    \item[$ii)$] $X_i=M[\rho^{i-m}\alpha]$, with $i\in \{1,\ldots,m\}$.
    
    \item[$iii)$] $Y_i=M[\beta\rho^{i-m}]$, with $i\in \{1,\ldots,m\}$.
    
    \item[$iv)$] $U_i=M[\beta\rho^{i-m}\alpha]$, with $i\in \{1,\ldots,m-1\}$. Moreover, we put $U_0=S(1)$.    
\end{enumerate}
The stable Auslander-Reiten quiver of $\Lambda$ is presented in Figure \ref{ARcomponent}.
\begin{figure}
{\tiny
$$\xymatrix{
& & & & & & & X_m \ar[d]^c & Y_m \ar[d]^c \ar@{-->}[l]_{\tau} & X_m \ar@{-->}[l]_{\tau}\\
& & & & & & U_{m-1} \ar[ru]^c \ar[d]^c & V_m \ar[ru]_h \ar[d]^c \ar@{-->}[l]_{\tau} & U_{m-1} \ar[ru]^c \ar@{-->}[l]_{\tau} & \\\
& & & & & X_{m-1} \ar[ru]^h & Y_{m-1} \ar[ru]_c \ar[d]^c \ar@{-->}[l]_{\tau} & X_{m-1} \ar[ru]^h \ar@{-->}[l]_{\tau} & & \\
& & & & \ \rotatebox{45}{$\cdots$} \ar[ru]^c & \ \rotatebox{45}{$\cdots$} \ar[ru]_h & \ \rotatebox{45}{$\cdots$} \ar[ru]_h & & & \\
& & & X_2 \ar[ru]^h \ar[d]^c & Y_2 \ar[ru]_c \ar[d]^c \ar@{-->}[l]_{\tau} & X_2 \ar[ru]^h \ar@{-->}[l]_{\tau} & & & & \\
& & U_1 \ar[ru]^c \ar[d]^c & V_2 \ar[ru]_h \ar[d]^c \ar@{-->}[l]_{\tau} & U_1 \ar[ru]^c \ar@{-->}[l]_{\tau} & & & & & \\
& X_1 \ar[ru]^h \ar[d]_c & Y_1 \ar[ru]_c \ar[d]^c \ar@{-->}[l]_{\tau} & X_1 \ar[ru]^h \ar@{-->}[l]_{\tau} & & & & & & \\
U_0 \ar[ru]^c & V_1 \ar[ru]_h \ar@{-->}[l]_{\tau} & U_0 \ar[ru]^c \ar@{-->}[l]_{\tau} & & & & & & & }$$
}
\caption{The stable Auslander-Reiten quiver for \textbf{(AE3)}.}
\label{ARcomponent}
\end{figure}
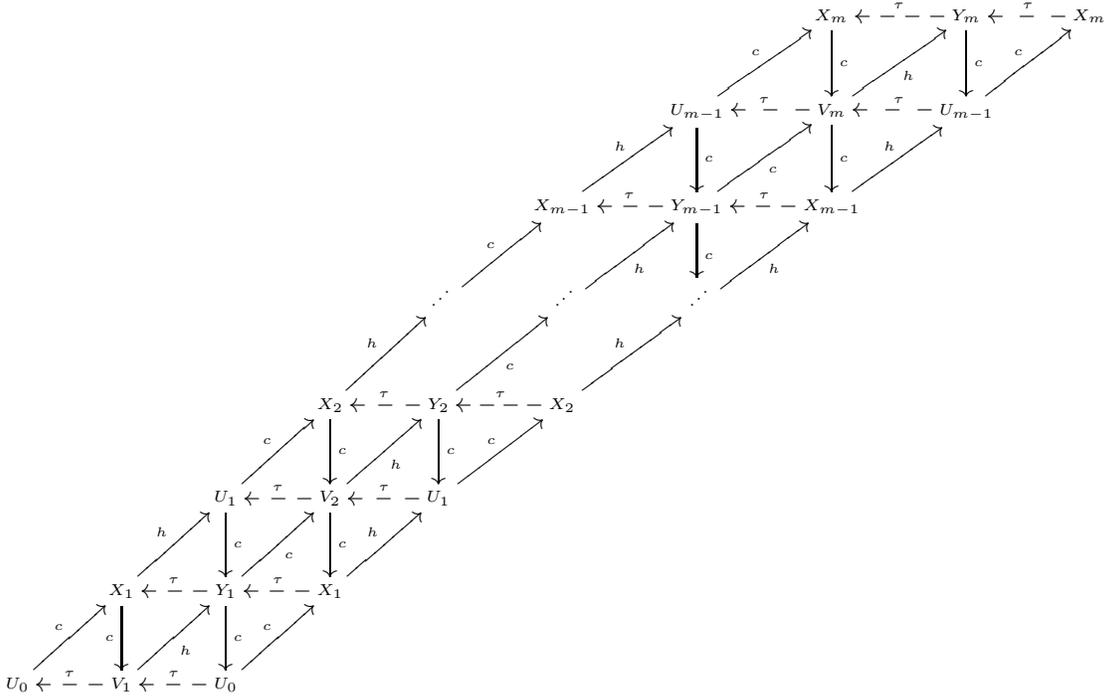
Observe that $U_i=(X_i)_h$, $U_i=(X_{i+1})_c$, $Y_i=(V_{i+1})_c$, and $U_i=\tensor[_h]{Y}{_i}$, for $1\leq i<m$, and $Y_i=(V_i)_h$ and $X_i=\tensor[_h]{V}{_i}$ for $1\leq i\leq m$ are all the modules corresponding to hooks and co-hooks of strings. Notice also that we have the following equalities for the operator $\Omega$:
$$\Omega(V_i)=Y_{m-i+1},\ \Omega(Y_i)=U_{m-i},\ \Omega(U_i)=X_{m-i},\ \Omega(X_i)=V_{m-i+1},$$
$$
\Omega^2(V_i)=U_{i-1},\ \Omega(Y_i)^2=X_i,\ \Omega^2(U_i)=V_i\mbox{ and } \Omega^2(X_i)=Y_i.
$$
\begin{prop}\label{terd3} Let $M$ be a non-projective indecomposable module over $\Lambda$. Then, the following statements hold.
\begin{enumerate}
    \item[$i)$] $\send_{\Lambda}(M)\simeq \Bbbk$ if and only if $M$ is isomorphic of one of the modules $U_0$, $V_1$, $X_m$, $Y_m$, $U_{m-1}$, $V_m$, $X_1$, $Y_1$.
    \item[$ii)$] Suppose that $\send_{\Lambda}(M)\simeq \Bbbk$. Then,
\begin{equation*}
\Ext^1_{\Lambda}(M,M)= \begin{cases}
0 & \text{if $M$ is isomorphic to one of the modules $U_0$, $V_1$, $X_m$, $Y_m$,}\\
\Bbbk & \text{if $M$ is isomorphic to one of the modules $U_{m-1}$, $V_m$, $X_1$, $Y_1$,}
\end{cases}
\end{equation*}
     and the universal deformation ring of $M$ is given by
\begin{equation*}
R(\Lambda,M)= \begin{cases}
\Bbbk & \text{if $M$ is isomorphic to one of the modules $U_0$, $V_1$, $X_m$, $Y_m$,}\\
\Bbbk \llbracket x \rrbracket /\langle x^m\rangle & \text{if $M$ is isomorphic to one of the modules $U_{m-1}$, $V_m$, $X_1$, $Y_1$.}
\end{cases}
\end{equation*}

\end{enumerate}
\end{prop}

\begin{proof} \mbox{}\\
\textit{Step 1.} First, we determine the indecomposable $\Lambda$-modules $M$ for which $\send_{\Lambda}(M)\cong \Bbbk$. We will prove that such modules are $S(0)$, $S(1)$, $X_1$, $Y_1$, $V_1$, $X_m$, $Y_m$ and $U_{m-1}$, that is the vertices highlighted with red and blue colors in the stable Auslander-Reiten quiver of $\Lambda$ in Figure \ref{ARcomponent}, or equivalently, the modules in the $\Omega$-orbits of the mouths of ${}_{s}\Gamma_{\Lambda}$.

First of all, since $\send_{\Lambda}(S)\cong \Bbbk$ for every simple module $S$, we have that the modules $U_0$, $X_1$, $Y_1$, $V_1$, $U_{m-1}$, $X_m$, $Y_m$ satisfy the desired isomorphism because they are in the $\Omega$-orbit of the simple modules $S(0)$ and $S(1)$.

It only remains to verify that $\send_{\Lambda}(V_i)$ is not isomorphic to $\Bbbk$ for $1<i<m$ (the other modules are in the $\Omega$-orbit of some $V_i$). Let $\phi$ be the canonical homomorphism
$$\phi:\xymatrix{
V_i \ar[rr] & & V_{i+1} \ar[rr] & & V_i}.$$
Then, $\phi$ does not factor trough any projective module, because the unique projective module $P$ satisfying the conditions in subsection \ref{subsec:string} to get a canonical homomorphism $V_i\rightarrow P$ is $P(0)$. Thus, the unique string module $N$ giving a canonical homomorphism $P\rightarrow N\rightarrow V_i$ is $N=V_i$, but the following diagram is not commutative
$$\xymatrix{
V_i \ar[rr] \ar[dd] & & V_{i+1} \ar[rr] & & V_i\\
&&&&\\
V_{i+1 } \ar[rr] & & P(0) \ar[rr] & & V_i \ .\ar[uu]}$$
Therefore, $\dim \send_{\Lambda}(V_i)>1$ and then $\send_{\Lambda}(V_i)$ is not isomorphic to $\Bbbk$.

\noindent
\textit{Step 2.} Now, we determine the tangent spaces $\Ext^1_{\Lambda}(M,M)$ provided that $\send_{\Lambda}(M)\cong \Bbbk$, that is, for $M$ being $U_0$, $X_1$, $Y_1$, $V_1$, $U_{m-1}$, $X_m$, $Y_m$, or $V_m$. Since $\Lambda$ is self-injective, we have $\Ext^1_{\Lambda}(M,M)\cong \shom_{\Lambda}(\Omega(M),M)$ and we just consider the cases $M=S(0)$ and $S(1)$, because $V_1$, $X_m$ and $Y_m$ are in the $\Omega$-orbit of $U_0=S(1)$ and $X_1$, $Y_1$, $U_{m-1}$ are in the $\Omega$-orbit of $S(0)$.

For the first case,
$$\Ext^1_{\Lambda}(S(0),S(0))\cong \shom_{\Lambda}(\Omega(S(0)),S(0))=\shom_{\Lambda}(Y_1,S(0))\cong \Bbbk,$$
where the last equality holds because $\shom_{\Lambda}(Y_1,S(0))$ is generated by the canonical projection.

On the other hand,
$$\Ext^1_{\Lambda}(S(1),S(1))\cong \shom_{\Lambda}(\Omega(S(1)),S(1))=\shom_{\Lambda}(X_m,S(1))=0,$$
where the last equality holds because there is no non-zero homomorphisms from $X_m$ to $S(1)$. Consequently, $\Ext^1_{\Lambda}(M,M)\cong \Bbbk$ for $M=S(1),V_1,X_m,Y_m$, and $\Ext^1_{\Lambda}(M,M)\cong \Bbbk$ for $M=S(0),X_1,Y_1,U_{m-1}$. Hence, by the results in subsection \ref{subsec:def}, the universal deformation ring of $S(1)$, $V_1$, $X_m$ and $Y_m$ is $\Bbbk$.

In the other cases, it is sufficient to calculate the universal deformation ring of $S(0)$. For this, we will apply a result in \cite[Thm. 1.1]{rv}. We define the sequence $\mathcal{L}$ and the homomorphisms $\epsilon_l$, $\iota_l$ and $\sigma_l$ in the following way.
\begin{enumerate}
    \item[$i)$] $\mathcal{L}=\{V_m,V_{m-1},V_{m-2},\ldots,V_1\}$.
    \item[$ii)$] For $1\leq l<t$, the homomorphisms 
    $\epsilon_l:V_{m-l}\twoheadrightarrow V_{m-l+1}$ is defined to be the canonical projection and $\iota_l:V_{m-l+1}\hookrightarrow V_{m-l}$ is defined to be the natural inclusion.

 \item[$iii)$] $\sigma_l:=\iota_l \epsilon_l$.
\end{enumerate}

Next, we will verify that the hypothesis of Theorem \ref{anote} holds for the simple $\Lambda$-module $S(0)$ joint with the sequence $\mathcal{L}$ and the homomorphisms $\sigma_l$.
\begin{enumerate}
    \item[$i)$] From the definition of $\sigma_l$, we have $\ker(\sigma_l)=S(0)$.
    \item[$ii)$] $\im(\sigma_l)=V_{m-l+1}$, $\im(\sigma_l^2)=V_{m-l+2}$, $\im(\sigma_l^3)=V_{m-l+3}$, ... In general, $\im(\sigma_l^n)=V_{m-l+n}$ for all $n\in \{1,\ldots,l\}$ and, in particular, $\im(\sigma_l^l)=V_m=S(0)$.
    \item[$iii)$] $\im(\sigma_l^{l+1})=\sigma_l(\im(\sigma_l^l))=\sigma_l(S(0))=0$.
    \item[$iv)$] $\dim \Hom_{\Lambda}(V_1,S(0))=1$ because $\Hom_{\Lambda}(V_1,S(0))$ is generated as vector space by the canonical projection.
    \item[$v)$] We have that $\Ext_{\Lambda}^1(V_1,S(0))=\shom_{\Lambda}(\Omega(V_1),S(0))=\shom_{\Lambda}(Y_m,S(0))$ which is $0$ because there is no any canonical homomorphism from $Y_m=M[\beta]$ to $S(0)$.

\end{enumerate}

We conclude that $R(\Lambda,M)=\Bbbk \llbracket x \rrbracket/\langle x^{|\mathcal{L}|}\rangle=\Bbbk \llbracket x \rrbracket/\langle x^m\rangle$ for $M=S(0)$, and then also for $M=M[\beta\rho^{-1}\alpha]$, $M=M[\rho^{1-m}\alpha]$ and $M=M[\beta\rho^{1-m}]$.
\end{proof}

\begin{rmk}\label{rmk:1}
It is important to point out that the classifications of universal deformation rings in Propositions \ref{tmmod}, \ref{tc2eser} and \ref{terd3} could be calculated using results by F. Bleher and D. Wackwitz in \cite{bw}, for a more general case. Indeed, these authors compute the universal deformation rings of modules over any Brauer tree algebra by using derived equivalences between these algebras and a certain class of self-injective Nakayama algebras. We have in the first case of the classification that $\Lambda$ is a Brauer tree algebra whose corresponding Brauer tree consists of an edge which is not a loop, one vertex $\zeta_1$ of multiplicity $1$ and one vertex $\zeta_0$ of multiplicity $m$. Graphically this Brauer tree corresponds to
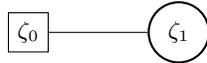
\begin{center}
\begin{tikzpicture}
        \tikzset{vnode/.style={draw,thick,circle,minimum width=8mm,inner sep=0pt}};
        \def\dist{2}
        
        \begin{scope}
            \node [minimum size=2mm,draw] (zeta0) {{\small $\zeta_0$}};

                \path (zeta0) --++ (0:\dist) node [vnode,draw] (zeta1){{\small $\zeta_1$}};
%
%
        \end{scope}
        
        
       \path
          (zeta0) 
                    edge node[left] {{\scriptsize \ \ }} (zeta1)
        ;
\end{tikzpicture}
\ \ \ \ \ \ \ \ \ \ \ \ \ \ \ \
\end{center}

Likewise, we have that in the second case of the classification that $\Lambda$ is a Brauer tree algebra whose associated Brauer tree is the following:
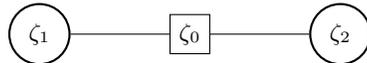
\begin{center}
\begin{tikzpicture}
        \tikzset{vnode/.style={draw,thick,circle,minimum width=8mm,inner sep=0pt}};
        \def\dist{2}
        
        \begin{scope}
            \node [minimum size=2mm,draw] (zeta0) {{\small $\zeta_0$}};

                \path (zeta0) --++ (0:\dist) node [vnode,draw] (zeta2){{\small $\zeta_2$}};
              \path (zeta0) --++ (180:\dist) node [vnode,draw] (zeta0) (zeta1) {{\small $\zeta_1$}};
        \end{scope}
        
        
       \path
          (zeta0) edge node[left] {{\scriptsize }} (zeta1)
                    edge node[left] {{\scriptsize \ \ }} (zeta2)
        ;
\end{tikzpicture}
\ \ \ \ \ \ \ \ \ \ \ \ \ \ \ \
\end{center}
where $\zeta_0$ has multiplicity $m\geq1$ and the other two vertices have multiplicity $1$. Finally, in the third case $\Lambda$ is a Brauer tree algebra whose associated Brauer tree is given as follows:
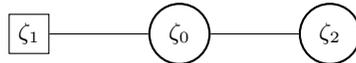
\begin{center}
\begin{tikzpicture}
        \tikzset{vnode/.style={draw,thick,circle,minimum width=8mm,inner sep=0pt}};
        \def\dist{2}
        
        \begin{scope}
            \node [vnode,draw] (zeta0) {{\small $\zeta_0$}};

                \path (zeta0) --++ (0:\dist) node [vnode,draw] (zeta2){{\small $\zeta_2$}};
              \path (zeta0) --++ (180:\dist) node [minimum size=2mm,draw] (zeta0) (zeta1) {{\small $\zeta_1$}};
        \end{scope}
        
        
       \path
          (zeta0) edge node[left] {{\scriptsize }} (zeta1)
                    edge node[left] {{\scriptsize \ \ }} (zeta2)
        ;
\end{tikzpicture}
\ \ \ \ \ \ \ \ \ \ \ \ \ \ \ \
\end{center}
where $\zeta_1$ has multiplicity $m\geq2$ and the other two vertices have multiplicity $1$.
While previous works by Bleher and Wackwitz provided a general framework for calculating universal deformation rings using derived equivalences, our focus here is on an explicit calculation for these two specific algebras, achieved through representation theory techniques without relying on derived equivalences.
\end{rmk}

\end{document}